\documentclass[12pt]{article}
\usepackage{amsthm,amssymb}

\usepackage{nccbbb}
\usepackage{bm}
\def\E23{{E_{\{2,3\}}}}

\usepackage{mathtools} 

{\theoremstyle{plain}%
  \newtheorem{theorem}{Theorem}
  \newtheorem{corollary}{Corollary}
  \newtheorem{proposition}{Proposition}
  \newtheorem{lemma}{Lemma}%
}
{\theoremstyle{remark}

}
{\theoremstyle{definition}

}
\newtheorem*{thmA}{Main Theorem}

\begin{document}
\renewcommand{\theequation}{\arabic{equation}}
\title{Multinomial Sum Formulas of Multiple Zeta Values}
\author{Kwang-Wu Chen\footnote{
Department of Mathematics, University of Taipei,
No. $1$,  Ai-Guo West Road, Taipei $10048$, Taiwan.
E-mail: kwchen@uTaipei.edu.tw}
}
\maketitle

\begin{abstract}  
For a pair of positive integers $n,k$ with $n\geq 2$, in this paper
we prove that
$$
\sum_{r=1}^k\sum_{|\bm\alpha|=k}{k\choose\bm\alpha}
\zeta(n\bm\alpha)=\zeta(n)^k
=\sum^k_{r=1}\sum_{|\bm\alpha|=k}
{k\choose\bm\alpha}(-1)^{k-r}\zeta^\star(n\bm\alpha),
$$
where $\bm\alpha=(\alpha_1,\alpha_2,\ldots,\alpha_r)$ 
is a $r$-tuple of positive integers.
Moreover, we give an application to combinatorics and get the following identity:
$$
\sum^{2k}_{r=1}r!{2k\brace r}=\sum^k_{p=1}\sum^k_{q=1}{k\brace p}{k\brace q}
p!q!D(p,q),
$$
where ${k\brace p}$ is the Stirling numbers of the second kind and $D(p,q)$ is the 
Delannoy number.
\end{abstract}

\noindent{\small {\it Key Words:} 
Multiple Zeta Values, Multiple Zeta-Star Values, Sum Formulas, 
Delannoy numbers, Fubini numbers.

\noindent{\it Mathematics Subject Classification 2010:}
Primary 11M32; Secondary 05A19.}

\setlength{\baselineskip}{18pt}
\section{Introduction}\label{sec.1}
The multiple zeta values (MZVs) are defined by 
\cite{BBBL2, Eie, Eie2, Gen}
$$
\zeta(\alpha_1,\alpha_2,\ldots,\alpha_r) = \sum_{1\leq k_1<k_2<\cdots<k_r}
k_1^{-\alpha_1}k_2^{-\alpha_2}\cdots k_r^{-\alpha_r}
$$
with positive integers $\alpha_1,\alpha_2,\ldots,\alpha_r$ 
and $\alpha_r\geq 2$ for the sake of convergence. 
The numbers $r$ and $|\bm\alpha|=\alpha_1+\alpha_2+\cdots+\alpha_r$
are the depth and weight of $\zeta(\bm\alpha)$. For our convenience,
we let $\{a\}^k$ be $k$ repetitions of $a$, for example, 
$\zeta(\{2\}^3)=\zeta(2,2,2)$.

MZVs of length one and two were already known to Euler. 
A systematic study of MZVs began in the early
1990s with the works of Hoffman \cite{Hof} and Zagier \cite{Zag}. 
Since then these numbers have emerged in several mathematical 
areas including algebraic geometry, Lie
group theory, advanced algebra, and combinatorics. 
Interestingly enough, MZVs also appear in theoretical physics, 
in particular in the context of perturbative quantum field
theory. This leads to a fruitful interplay between mathematics 
and theoretical physics. See \cite{Eie, Eie2, Zha} for introductory reviews.

A principal goal in the theoretical study of MZVs is to determine 
all possible algebraic relations among them. The most well-known identity
is the sum formula
$$
\sum_{|\bm\alpha|=n\atop \alpha_i\geq 1,\alpha_r\geq 2}
\zeta(\alpha_1,\ldots,\alpha_r)=\zeta(n).
$$
This formula was original obtained by Euler when $r=2$
and its general form was conjectured in \cite{Hof1}
and proved by Granville\cite{Gra} and Zagier\cite{Zag1}. 
Since then the sum formula
has been generalized and extended in various directions
\cite{CCE,ELO,Gen,GX,Ohn,OW,OZ}.

The multiple zeta-star values (MZSVs) are defined by 
\cite{IKOO,Mun,Zha}
$$
\zeta^\star(\alpha_1,\alpha_2,\ldots,\alpha_r) 
= \sum_{1\leq k_1\leq k_2\leq\cdots\leq k_r}
k_1^{-\alpha_1}k_2^{-\alpha_2}\cdots k_r^{-\alpha_r}
$$
with positive integers $\alpha_1,\alpha_2,\ldots,\alpha_r$ 
and $\alpha_r\geq 2$ for the sake of convergence. 

In this paper we prove another sum formula.
\begin{thmA}
For positive integers $n,k$ with $n\geq 2$, we have
$$
\sum_{r=1}^k\sum_{|\bm\alpha|=k}{k\choose\bm\alpha}
\zeta(n\bm\alpha)=\zeta(n)^k
=\sum^k_{r=1}\sum_{|\bm\alpha|=k}
{k\choose\bm\alpha}(-1)^{k-r}\zeta^\star(n\bm\alpha),
$$
where $\bm\alpha=(\alpha_1,\alpha_2,\ldots,\alpha_r)$ 
is a $r$-tuple of positive integers.
\end{thmA}

In additional, the stuffle product of two MVZs of depth $p$ and $q$ produces
$D(p,q)$ numbers of MZVs \cite{Chen1}, where $D(m,n)$
is the Delannoy number. The Delannoy number $D(m,n)$ is defined 
for nonnegative integers $m$ and $n$ by \cite[Page 81]{Com} 
$$
D(m,n)=\left\{\begin{array}{ll}1,&\mbox{if }m\cdot n=0,\\
D(m-1,n)+D(m-1,n-1)+D(m,n-1), &\mbox{if }m\cdot n\neq 0.\end{array}\right.
$$
The number $\sum^k_{r=1}r!{k\brace r}$ is usually called 
the Fubini numbers \cite{Com} or
the ordered Bell numbers \cite[Section 5.2]{Wil}
which count the number of weak orderings on a set of $k$ elements. 
Applying Main Theorem to combinatorics, we obtain an interesting identitiy
which connects Fubini numbers and Delannoy numbers:
$$
\sum^{2k}_{r=1}r!{2k\brace r}=\sum^k_{p=1}\sum^k_{q=1}{k\brace p}{k\brace q}
p!q!D(p,q),
$$
where $k$ is a positive integer and ${k\brace p}$ is the Stirling numbers of the second kind.

Our paper is organized as follows. In Section 2, we present
some preliminaries. In Section 3, we prove the main theorem. 
We give the combinatorial application in Section 4.
In the final section, we give a further application to the mutliple Hurwitz zeta functions
and the multiple Hurwitz zeta star functions.
%
%
\section{Some Preliminaries}\label{sec.2}
We summarize the algebraic setup of MZVs introduced by 
Hoffman \cite{Hof} and MZSVs introduced by Muneta \cite{Mun}
as follows.
Let us consider the coding of multi-indices
$\vec{s}=(s_1,\ldots,s_k)$, $s_i$ are positive integers and $s_k>1$, by words (that is,
by monomials in non-commutative variables) over the alphabet $X=\{x,y\}$ by
the rule
$$
  \vec{s}\mapsto x_{\vec{s}}=x^{s_1-1}yx^{s_2-1}y\cdots x^{s_k-1}y.
$$
We set
$$
  \zeta(x_{\vec{s}}):=\zeta(\vec{s})
$$
for all admissible words (that is, beginning with $x$ and ending with $y$); then
the weight (or the degree) $|x_{\vec{s}}|:=|\vec{s}|$ coincides with the total degree of the monomial 
$x_{\vec{s}}$, whereas the length (or the depth) $l(x_{\vec{s}}):=l(\vec{s})$ is the degree with
respect to the variable $y$.

Let ${\mathbb Q}\langle X\rangle={\mathbb Q}\langle x,y\rangle $ 
be the $\mathbb Q$-algebra of polynomials in 
two non-commutative variables which is graded by the degree (where each of the variables
$x$ and $y$ is assumed to be of degree $1$); we identify the algebra 
${\mathbb Q}\langle X\rangle $ 
with the graded $\mathbb Q$-vector space $\mathfrak H$ spanned by the monomials in the variables
$x$ and $y$ \cite{Hof}. 

We also introduce the graded $\mathbb Q$-vector spaces
${\mathfrak H}^1={\mathbb Q\bm 1}
\bigoplus {\mathfrak H}y$ and 
${\mathfrak H}^0={\mathbb Q\bm 1}
\bigoplus x{\mathfrak H}y$, where $\bm 1$ denotes the unit
(the empty word of weight $0$ and length $0$) of the algebra ${\mathbb Q}\langle X\rangle $.
Then the space ${\mathfrak H}^1$ can be regarded as the subalgebra of ${\mathbb Q}\langle X\rangle $ 
generated by the words $z_s=x^{s-1}y$, whereas ${\mathfrak H}^0$ is the $\mathbb Q$-vector
space spanned by all admissible words. 

Let us define two bilinear products $*$ (the 
{\it harmonic product}) and $\star$ on ${\mathfrak H}^1$ by the rules
$$
  \mbox{\boldmath 1}* w=w*\mbox{\boldmath 1}=w,\quad 
  \mbox{\boldmath 1}\star w=w\star\mbox{\boldmath 1}=w
$$
for any word $w$, and
\begin{eqnarray*}
  z_ju * z_k v &=& z_j(u*z_kv)+z_k(z_ju*v)+z_{j+k}(u*v),\\
  z_ju \star z_k v &=& z_j(u\star z_kv)+z_k(z_ju\star v)-z_{j+k}(u\star v)
\end{eqnarray*}
for any words $u$, $v$, any letters $x_i=x$ or $y$ ($i=1, 2$), 
and any generators $z_j$, $z_k$ of 
the subalgebra ${\mathfrak H}^1$, and then extend the above rules to the whole algebra ${\mathfrak H}$
and the whole subalgebra ${\mathfrak H}^1$ by linearity. 
It is known that each of the above products 
is commutative and associative. We denote the algebras
$({\mathfrak H}^1,+,*)$ and $({\mathfrak H}^1,+,\star)$ by
${\mathfrak H}^1_*$ and ${\mathfrak H}^1_\star$, respectively.
We define rational linear maps 
$\zeta:{\mathfrak H}^0_*\rightarrow\mathbb R$ and
$\zeta^\star:{\mathfrak H}^0_\star\rightarrow\mathbb R$
by $\zeta(\bm 1)=\zeta^\star(\bm 1)=1$ and
\begin{eqnarray*}
\zeta(z_{s_1}z_{s_2}\cdots z_{s_k}) &=&
\zeta(s_1,s_2,\ldots,s_k),\\
\zeta^\star(z_{s_1}z_{s_2}\cdots z_{s_k}) &=&
\zeta^\star(s_1,s_2,\ldots,s_k).
\end{eqnarray*}
Then these maps are algebra homomorphisms \cite{IKOO}:
$$
\zeta(w_1*w_2)=\zeta(w_1)\zeta(w_2),\qquad
\zeta^\star(w_1\star w_2)=\zeta^\star(w_1)\zeta^\star(w_2).
$$
%
%
\section{Main Theorem}\label{sec.3}
For our convenience, we let 
$$
\underbrace{z_n\cdots z_n}_k=z^k_n,\quad
\underbrace{z_n*z_n*\cdots*z_n}_k=z_{*n}^{k},\quad\mbox{and}\quad
\underbrace{z_n\star z_n\star\cdots\star z_n}_k=z_{\star n}^{k}.
$$
\begin{lemma} 
Let $n,r$ and $\alpha_1,\ldots,\alpha_r$ be positive integers, we have
\begin{eqnarray}\label{eq.07} 
\lefteqn{z_{n\alpha_1}\cdots z_{n\alpha_r}*z_n}\\
&=&z_nz_{n\alpha_1}\cdots z_{n\alpha_r}+\nonumber
z_{n\alpha_1}z_nz_{n\alpha_2}\cdots z_{n\alpha_r}+\cdots
+z_{n\alpha_1}\cdots z_{n\alpha_r}z_n \\ \nonumber
&&\qquad +z_{n(\alpha_1+1)}z_{n\alpha_2}\cdots z_{n\alpha_r}+
z_{n\alpha_1}z_{n(\alpha_2+1)}z_{n\alpha_3}\cdots z_{n\alpha_r}+
\cdots+z_{n\alpha_1}\cdots z_{n\alpha_{r-1}} z_{n(\alpha_r+1)},\\
\label{eq.08} 
\lefteqn{z_{n\alpha_1}\cdots z_{n\alpha_r}\star z_n}\\ \nonumber
&=&z_nz_{n\alpha_1}\cdots z_{n\alpha_r}+
z_{n\alpha_1}z_nz_{n\alpha_2}\cdots z_{n\alpha_r}+\cdots
+z_{n\alpha_1}\cdots z_{n\alpha_r}z_n \\ \nonumber
&&\qquad -\left(z_{n(\alpha_1+1)}z_{n\alpha_2}\cdots z_{n\alpha_r}+
z_{n\alpha_1}z_{n(\alpha_2+1)}z_{n\alpha_3}\cdots z_{n\alpha_r}+
\cdots+z_{n\alpha_1}\cdots z_{n\alpha_{r-1}} z_{n(\alpha_r+1)}\right).
\end{eqnarray}
\end{lemma}
\begin{proof}
We use induction on $r$ to prove Eq.\,(\ref{eq.07}). 
From the definition of the product $*$ we have the following.
\begin{eqnarray*}
\lefteqn{z_{n\alpha_1}\cdots z_{n\alpha_r}*z_n}\\
&=&z_{n\alpha_1}\left({z_{n\alpha_2}\cdots z_{n\alpha_r}*z_n}\right)
+z_nz_{n\alpha_1}\cdots z_{n\alpha_r}+z_{n(\alpha_1+1)}z_{n\alpha_2}
\cdots z_{n\alpha_r}.
\end{eqnarray*}
The induction hypothesis gives
\begin{eqnarray*}
\lefteqn{z_{n\alpha_2}\cdots z_{n\alpha_r}*z_n}\\
&=&z_nz_{n\alpha_2}\cdots z_{n\alpha_r}+
z_{n\alpha_2}z_nz_{n\alpha_3}\cdots z_{n\alpha_r}+\cdots
+z_{n\alpha_2}\cdots z_{n\alpha_r}z_n \\
&&\qquad +z_{n(\alpha_2+1)}z_{n\alpha_3}\cdots z_{n\alpha_r}+
z_{n\alpha_2}z_{n(\alpha_3+1)}z_{n\alpha_4}\cdots z_{n\alpha_r}+
\cdots+z_{n\alpha_2}\cdots z_{n\alpha_{r-1}} z_{n(\alpha_r+1)}.
\end{eqnarray*}
Substitute this identity in the above formula, we have the desired result.
Similarly we can prove Eq.\,(\ref{eq.08}) by induction, thus we omit it. 
\end{proof}
\begin{proposition}
Let $n,k$ be positive integers with $n\geq 2$. Then
\begin{eqnarray}\label{eq.09} 
z_{*n}^{k} &=& \sum^k_{r=1}\sum_{|\bm\alpha|=k}
{k\choose\alpha_1,\alpha_2,\ldots,\alpha_r}z_{n\alpha_1}
z_{n\alpha_2}\cdots z_{n\alpha_r},\\ \label{eq.10} 
z_{\star n}^{k} &=& \sum^k_{r=1}\sum_{|\bm\alpha|=k}
{k\choose\alpha_1,\alpha_2,\ldots,\alpha_r}(-1)^{k-r}z_{n\alpha_1}
z_{n\alpha_2}\cdots z_{n\alpha_r},
\end{eqnarray}
where $\bm\alpha=(\alpha_1,\alpha_2,\ldots,\alpha_r)$ is a $r$-tuple of positive integers.
\end{proposition}
\begin{proof}
We use induction on $k$ to prove Eq.\,(\ref{eq.10}).  
\begin{eqnarray*}
z_{\star n}^{k} &=& z_{\star n}^{k-1}\star z_n\\
&=&  \sum^{k-1}_{r=1}\sum_{|\bm\alpha|=k-1}
{k-1\choose\alpha_1,\alpha_2,\ldots,\alpha_r}(-1)^{k-1-r}z_{n\alpha_1}
z_{n\alpha_2}\cdots z_{n\alpha_r}\star z_n.
\end{eqnarray*}
The above identity is followed by the induction hypothesis.
By Lemma 1, we have 
\begin{eqnarray*}
z^k_{\star n}
&=& \sum^{k-1}_{r=1}\sum_{|\bm\alpha|=k-1}
{k-1\choose\alpha_1,\alpha_2,\ldots,\alpha_r}(-1)^{k-1-r}
\left(z_nz_{n\alpha_1}\cdots z_{n\alpha_r}+\cdots
+z_{n\alpha_1}\cdots z_{n\alpha_r}z_n\right. \\
&&\qquad\qquad\qquad
\left. -\left(z_{n(\alpha_1+1)}z_{n\alpha_2}\cdots z_{n\alpha_r}+
\cdots+z_{n\alpha_1}\cdots z_{n\alpha_{r-1}} z_{n(\alpha_r+1)}\right)\right).
\end{eqnarray*}
The former summand has $r+1$ numbers of $z_i$ and the latter summand has $r$
numbers of $z_i$ 
in each summation. We rewrite the summation such that each summand has the same numbers of $z_i$.
\begin{eqnarray*}
z^k_{\star n}
&=&(-1)^{k-1}z_{nk}+\\
&&\sum^{k-2}_{r=1}\left(\sum_{|\bm\alpha|=k-1}{k-1\choose\alpha_1,\ldots,\alpha_r}(-1)^{k-1-r}
(z_nz_{n\alpha_1}\cdots z_{n\alpha_r}+\cdots+z_{n\alpha_1}\cdots z_{n\alpha_r}z_n)\right.\\
&&\left.+\sum_{|\bm\alpha|=k-1}
{k-1\choose \alpha_1,\ldots,\alpha_{r+1}}(-1)^{k-1-r}
(z_{n(\alpha_1+1)}\cdots z_{n\alpha_r}
+\cdots+z_{n\alpha_1}\cdots z_{n(\alpha_{r+1}+1)})\right)\\
&&+k!\,z_n^k.
\end{eqnarray*}
We simplify the summation in the above identity as follows.
\begin{eqnarray*}
\lefteqn{\sum^{k-2}_{r=1}(-1)^{k-1-r}
\left(\sum_{|\bm\alpha|=k}
\left(\sum^{r+1}_{i=1}\frac{(k-1)!}{\alpha_1!\cdots\alpha_{i-1}!
(\alpha_i-1)!\alpha_{i+1}!\cdots\alpha_{r+1}!}\right)
z_{n\alpha_1}\cdots
z_{n\alpha_{r+1}}\right)}\\
&=&\sum^{k-2}_{r=1}(-1)^{k-1-r}
\left(\sum_{|\bm\alpha|=k}
{k-1\choose\alpha_1,\ldots,\alpha_{r+1}}
(\alpha_1+\cdots+\alpha_{r+1})
z_{n\alpha_1}\cdots
z_{n\alpha_{r+1}}\right)\\
&=&\sum^{k-2}_{r=1}(-1)^{k-1-r}
\left(\sum_{|\bm\alpha|=k}
{k\choose\alpha_1,\ldots,\alpha_{r+1}}
z_{n\alpha_1}\cdots
z_{n\alpha_{r+1}}\right).
\end{eqnarray*}
Combine the frist term $(-1)^{k-1}z_{nk}$ and the last term
$k!z_n^k$, we get our conclusion.
Eq.\,(\ref{eq.09}) can be proved by induction similarly.
\end{proof}
We transform Eq.\,(\ref{eq.09}) and Eq.\,(\ref{eq.10}) into the identities
among MZVs and MZSVs. 
\begin{theorem}
For positive integers $n,k$ with $n\geq 2$, we have
\begin{eqnarray*}
\zeta(n)^k &=& \sum^k_{r=1}\sum_{|\bm\alpha|=k\atop{\alpha_i\geq 1}}
{k\choose\alpha_1,\ldots,\alpha_r}\zeta(n\alpha_1,\ldots,n\alpha_r) \\
&=&\sum^k_{r=1}\sum_{|\bm\alpha|=k\atop{\alpha_i\geq 1}}
{k\choose\alpha_1,\ldots,\alpha_r}(-1)^{k-r}\zeta^\star(n\alpha_1,\ldots,n\alpha_r).
\end{eqnarray*}
\end{theorem}
%
%
\section{A Combinatorial Application}\label{sec.4}
Using Eq.\,(\ref{eq.09}) and counting the numbers of terms it happened, 
we obtain the following identity.
\begin{theorem}
Given a pair of positive integers $\ell,k$ with $1\leq \ell<k$, we have
$$
\sum^k_{r=1}r!{k\brace r}
=\sum^\ell_{p=1}\sum^{k-\ell}_{q=1}{\ell \brace p}{k-\ell\brace q}
p!q!D(p,q).
$$
\end{theorem}
\begin{proof}
There are 
$$
\sum^k_{r=1}\sum_{|\bm\alpha|=k\atop\alpha_i\geq 1}{k\choose\alpha_1,\alpha_2,\ldots,\alpha_r}
$$
terms in the right-hand side of Eq.\,(\ref{eq.09}). 
Since $\sum_{|\bm\alpha|=k\atop\alpha_i\geq 0}{k\choose\alpha_1,\ldots,\alpha_r}=r^k$, we have
$$
\sum_{|\bm\alpha|=k\atop\alpha_i\geq 1}{k\choose\bm\alpha}=
\sum_{|\bm\alpha|=k\atop\alpha_i\geq 1}{k\choose\alpha_1,\alpha_2,\ldots,\alpha_r}
=\sum^r_{j=1}(-1)^{r-j}{r\choose j}j^k
$$
by the inclusion--exclusion principle. By \cite[Eq.\,(6.19)]{GKP}, 
$$
\sum^r_{j=1}(-1)^{r-j}{r\choose j}j^k=r!{k\brace r},
$$
we can write the above number as
\begin{equation}\label{eq.11}
\sum_{|\bm\alpha|=k\atop\alpha_i\geq 1}{k\choose\bm\alpha}=r!{k\brace r}.
\end{equation}
For $1\leq \ell< k$, 
\begin{eqnarray*}
\lefteqn{\sum^k_{r=1}\sum_{|\bm\alpha|=k}
{k\choose\bm\alpha}z_{n\alpha_1}
z_{n\alpha_2}\cdots z_{n\alpha_r}}\\
&=&z_{*n}^{k}=z_{*n}^\ell * z_{*n}^{k-\ell} \\
&=&\sum^\ell_{p=1}\sum_{|\bm\beta|=\ell}
{\ell\choose\bm\beta}z_{n\beta_1}
z_{n\beta_2}\cdots z_{n\beta_p}*
\sum^{k-\ell}_{q=1}\sum_{|\bm\lambda|=k-\ell}
{k-\ell\choose\bm\lambda}z_{n\lambda_1}
z_{n\lambda_2}\cdots z_{n\lambda_q}\\
&=&\sum^\ell_{p=1}\sum_{|\bm\beta|=\ell}
\sum^{k-\ell}_{q=1}\sum_{|\bm\lambda|=k-\ell}
{\ell\choose\bm\beta}
{k-\ell\choose\bm\lambda}z_{n\beta_1}
z_{n\beta_2}\cdots z_{n\beta_p}*
z_{n\lambda_1}z_{n\lambda_2}\cdots z_{n\lambda_q}.
\end{eqnarray*}
Since the stuffle product of two MVZs of depth $p$ and $q$ produces
$D(p,q)$ numbers of MZVs, we count the numbers of MVZs in the above identity
and then we have
$$
\sum^k_{r=1}\sum_{|\bm\alpha|=k}
{k\choose\bm\alpha}=\sum^\ell_{p=1}\sum_{|\bm\beta|=\ell}
\sum^{k-\ell}_{q=1}\sum_{|\bm\lambda|=k-\ell}
{\ell\choose\bm\beta}
{k-\ell\choose\bm\lambda}D(p,q).
$$
Combing Eq.\,(\ref{eq.11}) we conclude our result. 
\end{proof}

If we set $k$ being an even integer and let $\ell$ be half of $k$, then we have a
beautiful identity.
\begin{corollary}
Given a positive integer $k$, we have
$$
\sum^{2k}_{r=1}r!{2k\brace r}=\sum^k_{p=1}\sum^k_{q=1}{k\brace p}{k\brace q}
p!q!D(p,q).
$$
\end{corollary}
%
%
\section{Further Remarks}\label{sec.5}
The multiple Hurwitz zeta function and the multiple Hurwitz zeta star function are defined by
\begin{eqnarray*}
\zeta(\alpha_1,\ldots,\alpha_r;x) &=&
\sum_{0\leq k_1<\cdots<k_r}(k_1+x)^{-\alpha_1}(k_2+x)^{-\alpha_2}\cdots
(k_r+x)^{-\alpha_r}\quad\mbox{and}\\
\zeta^\star(\alpha_1,\ldots,\alpha_r;x) &=& 
\sum_{0\leq k_1\leq\cdots\leq k_r}(k_1+x)^{-\alpha_1}(k_2+x)^{-\alpha_2}\cdots
(k_r+x)^{-\alpha_r}
\end{eqnarray*}
where $x$ is a positive real number, respectively.
Recently the author \cite{Chen2} gave explicit evaluations of $\zeta(\{m\}^n;1/2)$,  
$\zeta^{\star}(\{m\}^n;1/2)$, and also proved
\begin{eqnarray*}
\sum_{|\bm\alpha|=n}\zeta(m\alpha_1,\ldots,m\alpha_r;x)
&=&\sum_{p+q=n}(-1)^{p-k}{p\choose k}\zeta(\{m\}^p;x)\zeta^\star(\{m\}^q;x),\quad\mbox{and}\\
\sum_{|\bm\alpha|=n}\zeta^\star(m\alpha_1,\ldots,m\alpha_r;x)
&=&\sum_{p+q=n}(-1)^{p}{q\choose k}\zeta(\{m\}^p;x)\zeta^\star(\{m\}^q;x).
\end{eqnarray*}

We extend the original maps $\zeta$ and $\zeta^\star$ 
to rational linear maps 
$\zeta_x:{\mathfrak H}^0_*\rightarrow\mathbb R$ and
$\zeta_x^\star:{\mathfrak H}^0_\star\rightarrow\mathbb R$
by $\zeta_x(\bm 1)=\zeta_x^\star(\bm 1)=1$ and
\begin{eqnarray*}
\zeta_x(z_{s_1}z_{s_2}\cdots z_{s_k}) &=&
\zeta(s_1,s_2,\ldots,s_k;x),\\
\zeta_x^\star(z_{s_1}z_{s_2}\cdots z_{s_k}) &=&
\zeta_x^\star(s_1,s_2,\ldots,s_k;x).
\end{eqnarray*}
Then these maps are algebra homomorphisms:
$$
\zeta_x(w_1*w_2)=\zeta(w_1;x)\zeta(w_2;x),\qquad
\zeta_x^\star(w_1\star w_2)=\zeta^\star(w_1;x)\zeta^\star(w_2;x).
$$
Applying Eq.\,(\ref{eq.09}) and Eq.\,(\ref{eq.10}) to these two maps, we have
another sum formulas related
the multiple Hurwitz zeta function and the multiple Hurwitz zeta star function.
\begin{theorem}
For a positive real number $x$, and positive integers $n,k$ with $n\geq 2$, we have 
$$
\sum^k_{r=1}\sum_{|\bm\alpha|=k}
{k\choose\bm\alpha}\zeta(n\bm\alpha;x) 
=\zeta(n;x)^k =\sum^k_{r=1}\sum_{|\bm\alpha|=k}
{k\choose\bm\alpha}(-1)^{k-r}\zeta^\star(n\bm\alpha;x),
$$
where $\bm\alpha=(\alpha_1,\alpha_2,\ldots,\alpha_r)$ is a $r$-tuple of positive integers.
\end{theorem}

Recently, the special value $x=1/2$ of $\zeta(\bm\alpha;x)$ 
and $\zeta^\star(\bm\alpha;x)$ cause
many researchers for concern \cite{Hof2, SC1, Zhao}. 
Note that the multiple $t$ value and the multiple $t$ star value are usually defined by
\begin{eqnarray*}
t(\alpha_1,\ldots,\alpha_r)&=&\sum_{1\leq k_1<\cdots<k_r}
(2k_1-1)^{-\alpha_1}(2k_2-1)^{-\alpha_2}\cdots(2k_r-1)^{-\alpha_r},\quad\mbox{and}\\
t^\star(\alpha_1,\ldots,\alpha_r)&=&\sum_{1\leq k_1\leq\cdots\leq k_r}
(2k_1-1)^{-\alpha_1}(2k_2-1)^{-\alpha_2}\cdots(2k_r-1)^{-\alpha_r}.
\end{eqnarray*}
Hence 
$2^{|\bm\alpha|}t(\bm\alpha)=\zeta(\bm\alpha;1/2)$ and 
$2^{|\bm\alpha|}t^\star(\bm\alpha)=\zeta^\star(\bm\alpha;1/2)$.
We substitute $x=1/2$ in the above theorem, we have
$$
\sum^k_{r=1}\sum_{|\bm\alpha|=k}
{k\choose\bm\alpha}t(n\bm\alpha) 
=t(n)^k=\sum^k_{r=1}\sum_{|\bm\alpha|=k}
{k\choose\bm\alpha}(-1)^{k-r}t^\star(n\bm\alpha).
$$
%


%
%
%

\end{document}